\documentclass[11pt]{article}
\usepackage{epic,latexsym}
\usepackage{amssymb,amsmath,amsthm}
\usepackage{color}
\usepackage{tikz}
\usepackage{enumitem}
\usepackage{authblk}

\newtheorem{thm}{Theorem}
\newtheorem{lem}[thm]{Lemma}

\newtheorem{cor}[thm]{Corollary}

\newtheorem{prop}[thm]{Proposition}
\newtheorem{prm}[thm]{Problem}
\newtheorem{conj}{Conjecture}

\def \nmr {\begin{enumerate}}
\def \enmr {\end{enumerate}}

\def \tmz {\begin{itemize}}
\def \etmz {\end{itemize}}

\newcommand{\ml}{l\kern-0.55mm\char39\kern-0.3mm}

\begin{document}

\title{ On specific factors in graphs}

\author{Csilla Bujt\' as}

\affil{\small Faculty of Mathematics and Physics, University of Ljubljana, Slovenia}

\author{
Stanislav Jendro{\ml}\,}

\affil{Institute of Mathematics, P.~J.~\v Saf\' {a}rik University, Jesenn\'{a} 5, 040 01 Ko\v sice, Slovakia}

\author{Zsolt Tuza}

\affil{Alfr\'ed R\'enyi Institute of Mathematics, Budapest,
 and University of Pannonia, Veszpr\'{e}m, \vspace{-4ex} Hungary}

	
	
\date{}
\maketitle
	
\abstract{
 It is well known that if $G = (V, E)$} is a connected multigraph and $X\subset V$ is a subset of even order,
then $G$ contains a spanning forest $H$ such that each vertex from $X$ has an odd degree in $H$ and all the other vertices have an even degree in $H$. This spanning forest may have isolated vertices. If this is not allowed in $H$, then the situation is much more complicated. In this paper, we study this problem and generalize the concepts of even-factors and odd-factors in a unified form. 

\section{Notation and Terminology}

Let us first present some of the basic definitions, notations and terminology used in this paper. Other terminology will be introduced as it naturally occurs in the text or is used according to West's book \cite{West-2001}. We denote the vertex set and the edge set of a graph $G$ by $V(G)$ and $E(G)$, respectively.

Throughout this paper we use the term \emph{graph} in the general sense where both loops and multiple edges are allowed, hence cycles of length one (loop) or two (a pair of parallel edges) may also occur. A \emph{simple graph} is a graph having no loops or multiple edges.



The \emph{degree} of a vertex $v$, denoted by $\mathrm{deg}_G(v)$ or simply by $\mathrm{deg}(v)$ when the underlying graph is understood, is the number of edges incident with the vertex, where any loop is counted twice. The minimum degree in a graph $G$ will be denoted by $\delta(G)$ and the maximum degree by $\Delta(G)$. A graph is \emph{$r$-regular} if the degree of each vertex in $G$ is $r$, and the graph is \emph{regular} if it is $r$-regular for some $r$.
A set of edges in $G$ is a \emph{matching} if no two of them share a vertex. A \emph{perfect matching} (or \emph{1-factor}) in $G$ is a matching the edges of which span $G$.

\section{Introduction}

Given a graph $G$, we shall use the term \emph{p-factor} for a subgraph $H \subseteq G$ if
  $H$ is a spanning subgraph and has minimum degree $\delta(H) \geq 1$.
 A p-factor will also be referred to as a set of edges from $G$ that cover all the vertices of $G$.
 The letter p is intended to emphasize that all degrees are required to be positive, as opposed to the standard terms of factors and spanning subgraphs.

There is a very rich literature concerning factors of graphs, starting with the famous work of Petersen \cite{Pet-1891}. Several nice survey papers on this subject written by Chung and Graham \cite{ChGr-1981}, Akiyama and Kano \cite{AkKa-1985}, Volkmann \cite{Vol-1995}, and Plummer \cite{Plu-2007}, and the book of Akiyama and Kano \cite{AkKa-2011} together cover results of over one thousand papers.
Beyond the study of $1$-factors and $2$-factors in regular graphs as initiated in \cite{Pet-1891}, generalizations include $k$-factors, path-factors, even-factors, odd-factors, and more,
culminating in the ``Parity $(g,f)$-Factor Theorem'' proved by Lov\'{a}sz \cite{Lov-1970}.

The most general notion dealing with prescribed degrees for the vertices independently of each other is \emph{$B$-factor}, where a graph $G=(V,E)$ is given together with sets $B_v$ of nonnegative integers for its vertices, and one asks for a spanning subgraph $F$ such that $\deg_F(v)\in B_v$ holds for all $v\in V$.
Regarding the algorithmic complexity of this problem, Cornu\'ejols \cite{Corn-1988} proved the following important result.

\begin{thm}   \label{B-fact}
 There is an algorithm of running time\/ $O(n^4)$ which solves the\/ $B$-factor problem for any instance\/ $(G,\{B_v\mid v\in V\})$ on graphs\/ $G$ of order\/ $n$, provided that each\/ $B_v$ satisfies the following property: if an integer\/ $k\notin B_v$ is in the range\/ $\min(B_v)<k<\max(B_v)$, then both\/ $k-1$ and\/ $k+1$ are in\/ $B_v$.
\end{thm}

Connected factors, especially spanning trees, of specific properties have been extensively studied as well; see e.g.\ Chapter 8 in \cite{AkKa-2011} and surveys in the papers \cite{Has-2015}, 
\cite{OzY-2011}, and \cite{Tho-2008}. 
From that area we will employ the following result of Thomassen \cite{Tho-2008}.

\begin{thm} \label{thm:lowtree1} 
	Every\/ $2$-edge-connected graph\/ $G$ has a spanning tree\/ $T$ such that, for each vertex\/ $v$,\/ $\deg_T(v) \leq \frac{\deg_G(v) + 3}{2}$.
\end{thm}

In this paper we introduce a new concept which is the generalization of both, the even-factor and the odd-factor.

Let $G = (V, E)$ be a graph and let $X\subseteq V$ be a set of an even number of vertices. We say that a p-factor $H$ of $G$ is an \emph{X-parity-factor} of $G$ if $\deg_H(v)\equiv 1 \pmod{2}$ for every vertex $v \in X$, and $\deg_H(v)\equiv 0\pmod{2}$ for every $v \in V \setminus X$.
We emphasize that $\deg_H(v) > 0$ is required for all $v\in V$, by definition.

A graph $G = (V, E)$ has the \emph{strong parity property} if for every subset $X\subseteq V$ of an even number of vertices the graph has an $X$-parity-factor.
We give sufficient conditions for graphs to have this property, and formulate a related conjecture in Section~\ref{sec:4}.

Note that connectivity is an obvious necessary condition for the strong parity property, since an $X$ with $|X|=2$, having its two vertices from distinct components does not admit an $X$-parity-factor.
However, not every connected graph has this property, as we shall note at the beginning of the next section.
On the other hand, replacing the requirement of `p-factor' with `spanning subgraph', the necessary condition of connectivity becomes also sufficient, as shown by the following result\footnote{The existence of $H$ with the required parity properties easily follows by first selecting $r=|X|/2$ paths whose ends are mutually disjoint pairs of vertices from $X$, and then keeping exactly those edges for $H$ which occur in an odd number of the selected paths.
If a cycle $C\subset G$ violates the extra condition, then switching between selection and non-selection of its edges makes $|E(H)|$ decrease, without changing the parity of any $\deg_H(v)$.
Theorem \ref{thm:spanning1} later led to the development of the theory of $T$-joins; see e.g.\ Chapters 6.5 and 6.6 in \cite{LovPlu-1986}, or the survey \cite{Fra-1994}.} of Meigu Guan (whose name is also romanized as Mei-Ko Kwan).

\begin{thm}
	\label{thm:spanning1}
	If\/ $G$ is a connected graph and\/ $X \subseteq{V(G)}$ is an arbitrary subset of\/ $2r$ vertices of\/ $G$, then\/  $G$ has a spanning forest\/ $H$ such that
	\begin{itemize}[leftmargin=0.7in]		
	\item {$\deg_H(v) \equiv 1 \pmod{2}$ for any vertex\/ $v \in X$}.
	\item {$\deg_H(v) \equiv 0 \pmod{2}$ for any vertex\/ $v \in V(G) \setminus X$, where\/ $\deg_H(v) = 0$ is allowed}.
\end{itemize} 	
Moreover, in those subgraphs\/ $H$ of this kind which have minimum size, every cycle\/ $C\subset G$ has at most half of its edges in\/ $H$.
\end{thm} 


 
\section{The Strong Parity Property}
\label{sec:4}

It is a challenging problem to establish a nice general characterization for graphs satisfying the strong parity property. Hence, we concentrate on conditions which are necessary or sufficient for it.
First we mention some simple local obstructions, and also observe a complexity result.
Then we give some sufficient conditions for graphs to have the strong parity property. At the end we formulate a conjecture that can be considered as a strengthening of Theorem~\ref{thm:delta-4} and Theorem~\ref{thm:Ham} below, and prove it for 3-regular graphs.

\begin{prop} \label{prop:4}
 If a connected graph\/ $G=(V,E)$ contains any of the following, then it does not have the strong parity property:
 \begin{itemize}
   \item[$(i)$] a vertex\/ $v$ of degree\/ $1$;
   \item[$(ii)$] a path\/ $v_1v_2v_3$ with\/ $\deg_G(v_1)=\deg_G(v_2)=\deg_G(v_3)=2$,\/ $|V|>3$;
   \item[$(iii)$] a path\/ $v_1v_2v_3$ and a further vertex\/ $v_4$, such that\/ $\deg_G(v_1)=\deg_G(v_3)=2$,\/ $\deg_G(v_2)=3$,\/ $v_2v_4$ is a cut-edge of\/ $G$, and the component containing\/ $v_2$ in\/ $G-v_2v_4$ has order at least\/ $4$.
 \end{itemize}
\end{prop}

\begin{proof}
In each case we prescribe some vertices in and out of the set $X$, which will make it impossible to satisfy the parity conditions with a spanninng subgraph of all-positive degrees.
 \begin{itemize}
   \item[$(i)$] Just require $v\notin X$. This would need at least two edges incident with $v$.

   \item[$(ii)$] We prescribe $v_2\in X$ and $v_1,v_3\notin X$, plus a further vertex $w\in X$ distinct from $v_1,v_2,v_3$. Then an $X$-parity-factor $F$ would require all the four edges incident with $v_1$ and $v_3$, but then $v_2$ cannot have odd degree in $F$.

   \item[$(iii)$] Let $H$ be the component of $G-v_2v_4$ containing $v_4$. For each $v\in V(H)$ we prescribe $v\in X$ if and only if $\deg_H(v)$ is odd. Further, for the vertices in the component containing\/ $v_2$ in\/ $G-v_2v_4$ we set the conditions as in the preceding case $(ii)$.
   
   Suppose for a contradiction that there exists an $X$-parity-factor $F$ in $G$. Then $\deg_{F\cap E(H)}(v)\equiv\deg_H(v)$ (mod~2) holds for all $v\in V(H)\setminus\{v_4\}$. But then, since the number of odd degrees in $H$ --- as well as in $F\cap E(H)$ --- is even, the same congruence is valid for $v_4$, too. Consequently the edge $v_2v_4$ cannot occur in $F$. This leads to the contradiction that the restriction of $F$ to the subgraph induced by $V(G)\setminus V(H)$ would be a parity factor for $(ii)$.
   \vspace{-3.8ex}
 \end{itemize}
\end{proof}

 We say that a class ${\cal G}$ of graphs admits a \emph{forbidden induced subgraph characterization} if there is a (finite or infinite) class ${\cal F}$ of graphs such that a graph $G$ belongs to ${\cal G}$ if and only if $G$ contains no induced subgraph which is isomorphic to an $F\in {\cal F}$. The notion of \emph{forbidden subgraph characterization} is defined analogously. Proposition~\ref{prop:4} shows various possibilities for extending a graph $F$  to a graph $F'$ such that $F$ is an induced subgraph of $F'$ and the latter one does not satisfy the strong parity property. This directly implies the following statement.

\begin{cor}
 The class of graphs not having the strong parity property does not admit a forbidden (induced) subgraph characterization.
\end{cor}

A similar statement is true for the complementary class.

\begin{prop}
	The class of graphs having the strong parity property does not admit any forbidden (induced) subgraph characterization.
\end{prop}
\begin{proof}
	Given any candidate $F$ for a forbidden induced subgraph, we supplement $F$ with $|V(F)|$ new vertices such that every new vertex is a universal vertex (i.e., it is adjacent to all vertices) in the extended graph.
	Clearly $|V(F)| > 2$. We claim that this extended graph admits the strong parity property, despite that it contains $F$ as an induced subgraph. Let $X$ be an arbitrary given set of even size. If a vertex $v$ of $F$ has the same degree parity in the extended graph as prescribed by $X$, we keep all edges at $v$. For the other vertices of $F$ we delete a matching $M$ from their set to the set of new vertices. 
	Now consider the new vertices after the removal of $M$. Let $S$ be the set of vertices where the parity of current degree differs from what is prescribed by $X$. Note that also $S$ has even size, because the removal of each edge changes parity at exactly two vertices, and at the beginning (before the removal of $M$) we had an even number of odd degrees and also an even number of odd prescriptions by $X$, thus the symmetric difference of the two even sets was also even; this was modified by $-2$ or 0 or +2 by the removal of each matching edge. So, $|S|$ is even, and removing a perfect matching from the complete subgraph induced by $S$ we obtain an $X$-parity factor. Since we inserted more than two new vertices, the remaining graph after all the edge removals is still connected, and in particular all vertex degrees are positive.
\end{proof}

The definition of strong parity property puts a condition on exponentially many distributions of odd and even parities. 
For this reason, when just the formalization of the problem is considered, it is not trivial whether the corresponding decision problem belongs to any of the complexity classes NP and coNP. By definition, the problem of deciding whether a graph has a property ${\cal P}$ belongs to coNP, if and only if the decision problem of \emph{not} having property ${\cal P }$ belongs to NP.


\begin{thm}
 The decision problem, whether a generic input graph has the strong parity property, belongs to the class coNP.
\end{thm}

\begin{proof}
If $G=(V,E)$ does not have the strong parity property, then there is a subset $X\subseteq V$ for which no $X$-parity-factor exists. Calling for an NP-oracle we obtain an $X$ of this kind. 
Setting $B_v=\{k\mid 1\leq k\leq \deg(v), \, k\equiv 1 \ \mbox{\rm (mod~2)}\}$ for $v\in X$ and $B_v=\{k\mid 2\leq k\leq \deg(v), \, k\equiv 0 \ \mbox{\rm (mod~2)}\}$ for $v\in V\setminus X$, we can apply Theorem~\ref{B-fact} to verify in polynomial time that $X$ does not admit an $X$-parity-factor. By the same theorem a false solution can also be recognized efficiently.
\end{proof}

\begin{prm}
 Is the strong parity property checkable in polynomial time, or is it coNP-complete?
\end{prm}

The following theorem gives a sufficient condition for a graph to have the strong parity property.
 
\begin{thm} \label{thm:strongpar1}
	Let\/ $G$ be a connected graph of minimum degree\/ $\delta(G) \geq 2$. If\/ $G$ contains a connected p-factor\/ $F$ with\/
	$\deg_F(v) < \deg_G(v)$ for every vertex\/ $v$ of\/ $G$, then\/ $G$ has the strong parity property.
\end{thm} 

Before a proof of this theorem we introduce the concept of binary factor.  A sequence, whose elements are from the set $\{0, 1\}$ is called a \emph{binary sequence}.
 Let $G$ be a connected graph with vertex set $V(G) = \{v_1, \dots, v_n\}$ and degree sequence $D = \{d_1, \dots, d_n\}$, $d_i = \deg_G(v_i)$. The \emph{binary degree sequence} of $G$ is the binary sequence $A = \{a_1, \dots, a_n\}$, where
$a_i \equiv d_i \pmod{2}$. Clearly, the number of ones in $A$ is always even.

Let $B = \{b_1, \dots, b_n\}$ be a binary sequence with an even number of ones. A \emph{binary-factor} of G with respect to $B$ (or, equivalently, a \emph{$B$-factor}) is a p-factor $F$ of $G$, whose binary degree sequence is $B$.

\begin{lem}
Let\/ $G$ be a connected graph with vertex set\/ $V(G) = \{v_1, \dots, v_n\}$, with degree sequence\/ $\{d_1, \dots, d_n\}$, and with\/ $\delta(G) \geq 2$. Suppose further that\/ $G$ has a connected p-factor\/ $H$ with\/ $1 \leq \deg_H(v_i) < \deg_G(v_i)$ for all\/ $1\leq i\leq n$. Then, for every binary sequence\/ $B = \{b_1, \dots, b_n\}$ with an even number of ones, $G$ has a\/ $B$-factor\/ $F$. 
\end{lem}

\begin{proof}
Determine first the binary degree sequence $A = \{a_1, \dots, a_n\}$ of $G$. Next, compute the binary sequence $C = \{c_1, \dots, c_n\}$ with $c_i \equiv (a_i + b_i) \pmod{2}$ and define the set 
$X = \{v_i \mid  c_i = 1, i= 1, \dots, n\}$. It is easy to see that $X$ has an even number of elements. 
Now we apply Theorem~\ref{thm:spanning1} on the graph $H$ with the set $X$. The result is a spanning forest $K$ of $H$ with the binary sequence $C$. Then the required $B$-factor $F$ of $G$ is obtained by removing all edges of $K$ from the graph $G$. 
Here the conditions $K\subseteq H$ and $\deg_H(v_i) < \deg_G(v_i)$ guarantee that every vertex has a positive degree in $G-E(K)$.
\end{proof}

Now the proof of Theorem~\ref{thm:strongpar1} immediately follows from the lemma.
Below we give some classes of graphs for which the existence of a connected p-factor described in Theorem~\ref{thm:strongpar1} can be proved.

\begin{thm} \label{thm:delta-4}
If\/ $G$ is a\/ $2$-edge-connected graph with\/ $\delta(G) \geq 4$, then\/ $G$ has the strong parity property.
\end{thm}

\begin{proof}
We apply Theorem~\ref{thm:strongpar1} with $F$ being a spanning tree $T$ of $G$ as guaranteed by Theorem~\ref{thm:lowtree1}.
\end{proof}

\begin{thm} \label{thm:Ham}
If a graph\/ $G$ has a Hamiltonian path and\/ $\delta(G) \geq 3$, then it has the strong parity property.
\end{thm}
\begin{proof}
We apply Theorem~\ref{thm:strongpar1} with $F$ being a Hamiltonian path of $G$.
\end{proof}



\begin{thm}
	If every vertex of a connected graph\/ $G$ is incident with a\/ $2$-cycle or with a\/ $3$-cycle, then\/ $G$ has the strong parity property.
\end{thm}
\begin{proof}
	We start with the same line as in the proof of Theorem~\ref{thm:strongpar1}. Let $v_1, \dots , v_n$ be the vertices of $G$ and let $A = (a_1, \dots, a_n)$ be the binary degree sequence of $G$. For a subset $X\subseteq V(G)$ of even cardinality, first define the binary sequence $B=(b_1, \dots , b_n)$ where $b_i = 1$ if and only if $v_i \in X$. Then, consider the binary sequence 
$C = (c_1, \dots, c_n)$ with $c_i \equiv (a_i + b_i) \pmod{2}$ and take the set $Y = \{v_i \mid c_i = 1, i= 1, \dots, n\}$. 
	
For the graph $G$ and for the set $Y$, we consider a spanning subgraph $H$ which satisfies the parity conditions and has the smallest size $|E(H)|$ under this assumption. By Theorem~\ref{thm:spanning1}, there exists such a spanning subgraph $H$. We will prove that $\deg_H(v) < \deg_G(v)$ holds for every $v \in V(G)$.  First observe that, by the minimality assumption, $H$ does not contain parallel edges. Now, assume that there is a vertex $v$ such that $\deg_H(v) = \deg_G(v)$. This vertex cannot be incident with parallel edges in $G$ and hence, there is a triangle $uvu'$ in $G$.  Since $\deg_H(v) = \deg_G(v)$, both edges $uv$ and $u'v$ belong to $H$. If $uu' \in E(H)$, consider the spanning subgraph $H'$ with $E(H')=E(H) \setminus \{uv, u'v, uu'\}$; if $uu' \notin E(H)$, consider $H'$ with $E(H')=(E(H) \cup \{uu'\}) \setminus \{uv, u'v\}$. In either case, $H'$ satisfies the parity conditions and has strictly smaller size than $H$. This contradiction proves that $\deg_H(v) < \deg_G(v)$ for every $v\in V(G)$.	
	
Define the spanning subgraph $F$ of $G$ with $E(F)=E(G)\setminus E(H)$ and observe that $B$ is the binary sequence of $F$. Moreover, for every vertex $v$, $\deg_H(v) < \deg_G(v)$ implies $\deg_F(v) \ge 1$. Thus, $F$ is an $X$-parity factor of $G$.
\end{proof}

From this theorem we immediately have that all connected claw-free graphs with minimum degree at least $3$ have the strong parity property. In a more general form, we conclude the following.
\begin{cor}
If\/ $G$ is a connected\/ $K_{1, r}$-free graph with\/ $\delta(G) \geq r \geq 3$, then\/ $G$ has the strong parity property.
\end{cor}

We think that the following strengthening of Theorems~\ref{thm:delta-4} and \ref{thm:Ham} is also true.
\begin{conj}
Every\/ $2$-edge-connected graph of minimum degree at least three has the strong parity property.
\end{conj}

  To prove the conjecture for a graph $G$, it would be enough to find a p-factor
 $F\subset G$ mentioned in Theorem~\ref{thm:strongpar1}.
 However, the condition $\delta(G)\geq 3$ is not strong enough to ensure the
  existence of such a factor.
A general counterexample is the class of $3$-regular graphs having no Hamiltonian path.
Indeed, in those graphs any spanning tree contains a vertex of degree three because
 the graphs of maximum degree less than~3 are disjoint unions of paths and cycles.
On the other hand, for 3-regular graphs we can prove the conjecture, even in a
 slightly stronger form.

\begin{thm} \label{cubic}
    If\/ $G$ is a connected 3-regular graph such that the cut-edges of\/ $G$
 are contained in a path, then\/ $G$ has the strong parity property.
\end{thm}
\begin{proof}
By Petersen's theorem\footnote{The most famous form of Petersen's theorem states that every 2-connected 3-regular graph contains a 1-factor. However, the result proved in the original paper is stronger; namely, if a 3-regular graph does not admit a 1-factor, then it has at least three end-blocks. It means that the cut-edges cannot be included in a single path.}
 \cite{Pet-1891} $G$ has a 1-factor $M$, hence removing the
 edges of $M$ we obtain a 2-factor; let the components of $G-M$ be
  $H_1,\dots,H_k$.
Here each $H_i$ is a cycle, whose length can be any positive integer including
 1 (loop) or 2 (two parallel edges) also.
Since $G$ is connected, one can select a subset $F\subseteq M$ of $k-1$ edges
 from the perfect matching such that $H^+ := E(H_1) \cup \dots \cup E(H_k) \cup F$
  is a connected spanning subgraph of $G$.

Instead of $X$ we consider $Z:=V(G)\setminus X$. Note that also $Z$ has an even
 number of vertices, say $|Z|=2m$, because $G$ is 3-regular, hence $|V(G)|$ is even.
We are going to prove that $H^+$ admits a selection of $m$ paths, which we shall
 denote by $P^1,\dots,P^m$, such that they are
 mutually vertex-disjoint, all have both of their endpoints in $Z$, and all their
 internal vertices are in $X$.

 We proceed by induction on $k$.
If $k=1$, then $H^+$ is a Hamiltonian cycle in $G$, which is split into $2m$
 subpaths by the vertices of $Z$.
Selecting every second path we obtain a collection of paths as required.

Assume now $k>1$. There exists a cycle in $H^+$, say $H_k$, which is incident with
 precisely one edge of $F$.
Let this edge be $vw$, where $v\in V(H_k)$ and $w\in V(H_j)$ for some $j\neq k$.
We also set $Z_k:=Z\cap V(H_k)$.

If $|Z_k|$ is even and positive, then $Z_k$ splits $H_k$ into an
 even number of subpaths.
In this case we can select every second subpath, as we did in the case of $k=1$,
 delete $V(H_k)$ and all its incident edges from $H^+$, and apply induction.
(For $|Z_k|=0$ we just delete $V(H_k)$ and the incident edges.)

Suppose that $|Z_k|$ is odd.
 We now choose a vertex $z\in Z_k$ which is closest to $v$ along the cycle $H_k$.
  (The case of $z=v$ is also possible.)
If $|Z_k|>1$, we consider the shortest subpath $P$ of $H_k$ which is disjoint from
 $\{z,v\}$ and contains all vertices of $Z_k\setminus\{z\}$.
This $P$ is split into an odd number of subpaths by $Z_k\setminus\{z\}$; we select
 the first, third, ..., last of them.
After that, we apply the induction hypothesis to the graph obtained by the removal
 of $H_k$, for the modified set $Z':=(Z\setminus Z_k)\cup \{w\}$.
Note that $Z'$ contains an even number of vertices, say $2m'$, and the modified
 graph has a similar tree structure with a 2-factor consisting of $k-1$ cycles.
Hence it contains a collection of $m'$ paths whose set of endpoints is
 identical to $Z'$.
One of those paths ends in $w$;
 we extend it until $z$ using the shortest $v$--$z$ path in $H_k$.
This procedure proves that the required collection $P^1,\dots,P^m$ of $m$ paths
 exists indeed.

To complete the proof of the theorem we consider the graph $H^*$ with vertex set
 $V(G)$ and edge set $E(G)\setminus \left( \bigcup_{i=i}^m E(P^i) \right)$.
If a vertex $u$ is the endpoint of some $P^i$, then it has degree 2 in $H^*$;
 if it is an internal vertex of some $P^i$, then it has degree 1 in $H^*$;
 and if it is outside of $\left( \bigcup_{i=i}^m V(P^i) \right)$, then it
 has degree 3 in $H^*$.
This fact verifies the validity of the theorem because a vertex is an endpoint
 of some $P^i$ if and only if it belongs to $Z$.
\end{proof}

\bigskip
{\bf Acknowledgments}

\smallskip
The first author acknowledges the financial support from the Slovenian Research Agency under the project N1-0108.
This work of the second author was supported by the Slovak Research and Development Agency under the Contract No.\ APVV-19-0153.
Research of the third author was supported in part by the National Research,
 Development and Innovation Office -- NKFIH under the grant SNN 129364.
The authors would like to thank J\'{u}lius Czap for his helpful comments.

\end{document}